\documentclass[reqno]{amsart}

\usepackage{amsmath,amssymb,amsthm}

\usepackage{tikz}

\usepackage{caption}
\usepackage{graphicx}
\usepackage{graphics}

\newtheorem*{thm}{Theorem}
\newtheorem*{lemma}{Lemma}
\newtheorem*{corollary}{Corollary}

\begin{document}

\title[]{An Inequality Characterizing Convex Domains}

\author[]{Stefan Steinerberger}
\address{Department of Mathematics, University of Washington, Seattle, WA 98195, USA} \email{steinerb@uw.edu}

\thanks{S.S. is supported by the NSF (DMS-2123224) and the Alfred P. Sloan Foundation.}

\begin{abstract} A property of smooth convex domains $\Omega \subset \mathbb{R}^n$ is that if two points on the
boundary $x, y \in \partial \Omega$ are close to each other, then their normal vectors $n(x), n(y)$ point
roughly in the same direction and this direction is almost orthogonal to $x-y$ (for `nearby' $x$ and $y$). 
 We prove there exists a constant $c_n > 0$ such that 
if $\Omega  \subset \mathbb{R}^n$ is a bounded domain with $C^1-$boundary $\partial \Omega$, then
$$ \int_{\partial \Omega \times \partial \Omega}  \frac{\left|\left\langle n(x), y - x \right\rangle  \left\langle  y - x, n(y) \right\rangle   \right| }{\|x - y\|^{n+1}}~d \sigma(x) d\sigma(y) \geq c_n |\partial \Omega|$$
and equality occurs if and only if the domain $\Omega$ is convex. 
\end{abstract}
\maketitle

\section{Introduction and Result}
Let $\Omega \subset \mathbb{R}^n$ be a bounded domain with $C^1-$boundary. The normal vectors $n(x), n(y)$ of two elements of the boundary, $x,y \in \partial \Omega$,will point roughly in the same direction which is roughly orthogonal to $y-x$ if $x$ and $y$ are close. In regions of
large curvature the normal vector changes quickly but convex domains whose boundary has regions with large curvature are `flatter' in
other regions and it might all average out in the end. We prove a quantitative version of this notion.

\begin{thm} There exists $c_n > 0$ so that for any bounded $\Omega \subset \mathbb{R}^n$ with $C^1-$boundary
$$ \int_{\partial \Omega \times \partial \Omega}  \frac{\left|\left\langle n(x), y - x \right\rangle  \left\langle  y - x, n(y) \right\rangle   \right| }{\|x -y\|^{n+1}}~d \sigma(x) d\sigma(y) \geq c_n |\partial \Omega|$$
with equality if and only if the domain $\Omega$ is convex.
\end{thm}
Integration is carried out with respect to the $(n-1)-$dimensional Hausdorff measure
and the size of the boundary $ |\partial \Omega|$ is measured the same way. Somewhat to our surprise,
we were unable to find this statement in the literature. It can be interpreted as a global conservation law
for convex domains or as a geometric functional with an extremely large set of minimizers (all convex domains). The requirement of the boundary $\partial \Omega$ being $C^1$
can presumably be somewhat relaxed.

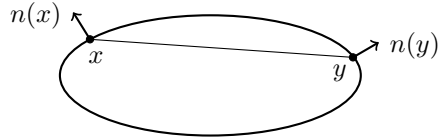
\begin{figure}[h!]
\begin{center}
\begin{tikzpicture}[scale=0.8]
\draw[thick]  (0,0) ellipse (2.5cm and 1cm);
\filldraw (-2,0.6) circle (0.06cm);
\filldraw (2.37,0.3) circle (0.06cm);
\draw (-2, 0.6) -- (2.37, 0.3);
\draw [thick, ->] (-2, 0.6) -- (-2.28, 1.05);
\draw [thick, ->] (2.37, 0.3) -- (2.8, 0.55);
\node at (-1.9, 0.3) {$x$};
\node at (2.15, 0.05) {$y$};
\node at (-2.9, 1) {$n(x)$};
\node at (3.4, 0.5) {$n(y)$};
\end{tikzpicture}
\end{center}
\caption{If $x$ and $y$ are close, then $n(x)$ and $n(y)$ are nearly orthogonal to $x-y$ unless $x$ and $y$ are far apart.}
\end{figure}

 If $\Omega$ is the unit ball in $\mathbb{R}^n$, then
$\partial \Omega = \mathbb{S}^{n-1}$ and for any $x, y \in \mathbb{S}^{n-1}$, we have $n(x) = x$ and
$ \|x-y\|^2 = 2 - 2\left\langle x, y \right\rangle$. This simplifies the expression since
$$ \left|\left\langle n(x), y - x \right\rangle  \left\langle  y - x, n(y) \right\rangle   \right| = (1 - \left\langle x,y \right\rangle)^2.$$
 Moreover, using rotational symmetry and
 $w = (1,0,0,\dots,0)$ for the north pole,
\begin{align*}
\int_{\mathbb{S}^{n-1} \times \mathbb{S}^{n-1}}  \frac{ (1 - \left\langle x, y \right\rangle)^2 }{\|y - x\|^{n+1}}~d \sigma(x) d\sigma(y) &=
 \int_{\mathbb{S}^{n-1} \times \mathbb{S}^{n-1}}  \frac{ (1 - \left\langle x, y \right\rangle)^2 }{(2-2\left\langle x, y\right\rangle)^{\frac{n+1}{2}}}~d \sigma(x) d\sigma(y) \\
 &= \frac{|\mathbb{S}^{n-1}|}{2^{\frac{n+1}{2}}} \int_{\mathbb{S}^{n-1} }  (1-\left\langle x, w\right\rangle)^{-\frac{n-3}{2}}~d \sigma(x) 
\end{align*}
which implies
\begin{align*}
 c_n = \frac{1}{2^{\frac{n+1}{2}}} \int_{\mathbb{S}^{n-1} }  (1-x_1)^{-\frac{n-3}{2}}~d \sigma(x)=  \frac{1}{2}   \int_{\mathbb{S}^{n-1}} \left| x_1 \right| d\sigma(x).
 \end{align*}
The constant has a simple form in low dimensions where $c_2 = 2$ and $c_3 = \pi$. 
Our proof can be best described as an application of Integral Geometry; we formulate and use a bilinear version of the Crofton formula. 
The proof tells us a little bit more: if the domain $\Omega$ is convex,  stronger statements can be made.
\begin{corollary} For any convex, bounded $\Omega \subset \mathbb{R}^n$ with $C^1-$boundary and all $x \in \partial \Omega$
$$ \int_{\partial \Omega}  \frac{\left|\left\langle n(x), y - x \right\rangle  \left\langle  y - x, n(y) \right\rangle   \right| }{\|x -y\|^{n+1}}~d\sigma(y)  = c_n.$$
If $x \in \Omega \setminus \partial \Omega$ and $w \in \mathbb{S}^{n-1}$ is an arbitrary unit vector, then
$$ \int_{\partial \Omega}  \frac{\left|\left\langle w, y - x \right\rangle  \left\langle  y - x, n(y) \right\rangle   \right| }{\|x -y\|^{n+1}}~d\sigma(y)  = 2 \cdot c_n.$$
\end{corollary}
We note that $c_n$ is the exact same constant as above (which can be seen by integrating the first equation over $\partial \Omega$ with respect to $d \sigma(x)$). Some of the conditions can presumably be relaxed a little. The Crofton formula is known to hold in a very general setting (see Santal\'o \cite{santa2}). It is an interesting question whether any of these results could be generalized to more abstract settings.

\section{Proof of the Theorem}
 The Crofton formula in $\mathbb{R}^n$ (see, for example, Santal\'o \cite{santa}) states that for rectifiable $S$ of co-dimension 1 one has
$$ |S| = \alpha_n \int_{L} n_{\ell}(S) d\mu(\ell),$$
where the integral runs over the space of all oriented lines in $\mathbb{R}^n$ with respect to the kinematic measure $\mu$ (which is invariant under all
rigid motions of $\mathbb{R}^n$) and $n_{\ell}(S)$ is the number of times the line $\ell$ intersects the surface $S$. The constant
$\alpha_n$ can be computed by picking $S = \mathbb{S}^{n-1}$ but will not be needed for our argument.

\begin{lemma}
Let $\Omega \subset \mathbb{R}^n$ be a bounded domain with $C^1-$boundary. Almost all lines $\ell$ (with respect to the kinematic measure) intersect the boundary $\partial \Omega$ either never or in exactly two points if and only if $\Omega$ is convex. \end{lemma}
\begin{proof}
If $\Omega$ is convex, the result is immediate. Suppose now $\Omega$ is not convex; then there exists a boundary point $x \in \partial \Omega$
such that the supporting hyperplane does not contain all of the domain on one side (note that because the boundary of $\Omega$ is $C^1$, the
supporting hyperplane is unique).
In particular, there exists $y \in \Omega$ that is on the other side of the supporting hyperplane. The line $\ell$ that goes through $x$ and $y$ satisfies
$n_{\ell}(\partial \Omega) \geq 4$, moreover, this is stable under some perturbations of the line (and thus a set of kinematic measure larger than 0) because $\partial \Omega$ is $C^1$. \end{proof}

\begin{proof}[Proof of the Theorem]
We first note for almost all lines $\ell$ (with respect to the kinematic measure) the number of intersections $n_{\ell}(\partial \Omega)$ is either 0 or at least 2: if a line enters the domain, it also has to exit the domain (lines that are tangential to the boundary are a set of measure 0). This implies
$$ |\partial \Omega| = \alpha_n  \int_{L} n_{\ell}(\partial \Omega) ~d \mu(\ell) \leq \frac{\alpha_n}{2}  \int_{L} n_{\ell}(\partial \Omega)^2 ~d \mu(\ell)$$
with equality if and only if $\Omega$ is convex.
At this point we pick a small $\varepsilon > 0$ and decompose the boundary
$ \partial \Omega = \bigcup_i  \partial \Omega_i$ into small disjoint regions that have diameter $\leq \varepsilon \ll 1$ (and $\varepsilon$ will later tend to 0).
Naturally,
$$ n_{\ell}(\partial \Omega)^2 = \left[ n_{\ell}\left( \bigcup_i  \partial \Omega_i\right) \right]^2 =  \left[ \sum_i n_{\ell}\left(  \partial \Omega_i\right) \right]^2= \sum_{i,j} n_{\ell}( \partial \Omega_i) n_{\ell}(\partial \Omega_j)$$
We will now evaluate the integral over such a product. The diagonal terms $i=j$ behave
a little bit differently than the non-diagonal terms and we start with those. As $\varepsilon \rightarrow 0$, the fact that the 
boundary is $C^1$ implies that a `random' line will hit any such infinitesimal segment at most once and thus the Crofton formula implies
\begin{align*}
 \frac{\alpha_n}{2}  \int_{L} \sum_i n_{\ell}( \partial \Omega_i)^2  ~d \mu(\ell) &= (1+o(1))  \frac{\alpha_n}{2}  \int_{L} \sum_i n_{\ell}( \partial \Omega_i)  ~d \mu(\ell)\\ 
 &= (1+o(1))  \frac{\alpha_n}{2}  \int_{L} n_{\ell}( \partial \Omega)  ~d \mu(\ell) =  (1+o(1))  \frac{|\partial \Omega|}{2}
 \end{align*}
 which is nicely behaved as $\varepsilon \rightarrow 0$ (the error could be made quantitative in terms of the modulus of continuity of the normal vector). It remains to analyze the off-diagonal terms.
 Let us assume that $\partial \Omega_x \subset \partial \Omega$ is a small
segment centered around $x \in \partial \Omega$ and  $\partial \Omega_y \subset \partial \Omega$ is a small
segment centered around $y \in \partial \Omega$ and that both are scaled to have surface area $0 < \varepsilon \ll 1$. We can
also assume, because the surface is $C^1$ and we are allowed to take $\varepsilon$ arbitrarily small, that they are approximately given by hyperplanes (and, as above, the error is a lower order term coming from curvature). 
The quantity to be evaluated, 
$$ \int_{L} n_{\ell}( \partial \Omega_x) n_{\ell}(\partial \Omega_y) d\mu(\ell), \qquad \mbox{can be seen in probabilistic terms}$$
as the likelihood that a `random' line (random as induced by the kinematic measure $\mu$) intersects
both $\partial \Omega_x$ and $\partial \Omega_y$. Appealing to the law of total probability
$$ \mathbb{P}\left(  n_{\ell}( \partial \Omega_x) n_{\ell}(\partial \Omega_y) = 1\right) = \mathbb{P}( n_{\ell}(\partial \Omega_y) = 1 \big|  n_{\ell}( \partial \Omega_x) =1) \cdot \mathbb{P}( n_{\ell}( \partial \Omega_x) =1).$$
The last quantity is easy to evaluate: by Crofton's formula
$$  \mathbb{P}( n_{\ell}( \partial \Omega_x) =1) = \frac{1}{\alpha_n} | \partial \Omega_x| = \frac{\varepsilon}{\alpha_n}.$$
It remains to compute the second term: the likelihood of a `random' line hitting $\partial \Omega_y$ provided that it has already hit $\partial \Omega_x$. For this purpose, we first consider what we can say about random lines that have hit $\partial \Omega_x$.  The distribution of $\partial \Omega_x \cap \ell$, provided it is not empty, is, to leading order, uniformly distributed over $\partial \Omega_x$ because $\partial \Omega_x$ is, to leading order, part of a hyperplane and the kinematic measure is translation-invariant.  In contrast, the direction $\phi$ of intersection (identified with unit vectors on $\mathbb{S}^{n-1}$) is not uniformly distributed: the likelihood is proportional to the size of the projection of $\Omega_x$ in direction of $\phi$ which is proportional to the inner product of $\phi$ with the normal vector $n(x)$. Hence the probability distribution of the direction of intersection $\phi$ of lines conditioned on hitting $\partial \Omega_x$ is given by
$$ \Psi(\phi) =   \frac{2 \left\langle n(x), \phi \right\rangle } {\int_{\mathbb{S}^{n-1}} \left|\left\langle w, n(x)\right\rangle \right| d\sigma(w) }, $$
where the factor $2$ comes from the fact that each line creates two directions of intersections. 
This allows us to perform a change of measure: we may assume that the lines are oriented uniformly at random provided that we later weigh the end result by $\Psi$.
If the lines are oriented in all directions uniformly, then it is easy to see the likelihood of hitting $\partial \Omega_y$ provided one has already hit $\partial \Omega_x$: it is simply proportional to the size of the projection of $\partial \Omega_y$ onto the sphere of radius $\|x-y\|$ centered at $x$. The projection shrinks
the area by a factor of $ \left| \left\langle n(y), (x-y)/\|x-y\| \right\rangle \right|$. The relative likelihood is then proprtional to
$$P= \Psi\left( \frac{x-y}{\|x-y\|}\right)   \left| \left\langle n(y), \frac{x-y}{\|x-y\|} \right\rangle \right| \frac{\varepsilon}{\|x-y\|^{n-1}}.$$
Plugging in the definition of $\Psi$ this simplifies to
$$ P =    2\left(   \int_{\mathbb{S}^{n-1}} \left|\left\langle w, n(x)\right\rangle\right| d\sigma(w)\right)^{-1}  \frac{\left| \left\langle n(x), x-y\right\rangle \left\langle x-y, n(y) \right\rangle \right|}{\|x-y\|^{n+1}}  \varepsilon,$$
where we note that, by rotational symmetry of the sphere, the first integral is actually independent of the direction in which $n(x)$ is pointing.
Altogether,
\begin{align*}
 |\partial \Omega|& = \alpha_n \int_{L} n_{\ell}(\partial \Omega) ~d \mu(\ell) \leq  \frac{\alpha_n}{2} \int_{L} n_{\ell}(\partial \Omega)^2 ~d \mu(\ell) \\
 &= \frac{\alpha_n}{2} \int_{L}  \sum_i n_{\ell}(\partial \Omega_i) ~d \mu(\ell) + \frac{\alpha_n}{2} \int_{L}  \sum_{i \neq j} n_{\ell}(\partial \Omega_i) n_{\ell}(\partial \Omega_j)  ~d \mu(\ell).
\end{align*}
The inequality is an equation if and only if $\Omega$ is convex. As already discussed above, the first term tends to $|\Omega|/2$ as $\varepsilon \rightarrow 0$. Thus, for arbitrary $a \in \mathbb{S}^{n-1}$,
\begin{align*}
\frac{ |\partial \Omega|}{2} &\leq \lim_{\varepsilon \rightarrow 0} \frac{\alpha_n}{2} \int_{L}  \sum_{i \neq j} n_{\ell}(\partial \Omega_i) n_{\ell}(\partial \Omega_j)  ~d \mu(\ell) \\
&=\left( \int_{\mathbb{S}^{n-1}} \left|\left\langle w, a\right\rangle\right| d\sigma(w)\right)^{-1} \int_{\partial \Omega \times \partial \Omega}  \frac{\left|\left\langle n(x), y - x \right\rangle  \left\langle  y - x, n(y) \right\rangle   \right| }{\|y - x\|^{n+1}}~d \sigma(x) d\sigma(y).
\end{align*}
This establishes the inequality with constant
$$ c_n =\frac{1}{2} \int_{\mathbb{S}^{n-1}} \left|\left\langle w, n\right\rangle\right| d\sigma(w) = \frac{1}{2}  \int_{\mathbb{S}^{n-1}} \left| w_1 \right| d\sigma(w).$$
\end{proof}

\section{Proof of the Corollary}
\begin{proof} The proof of the Corollary is using the same computation as the proof of the Theorem in two additional settings leading to the two identities. Let $\Omega$ be convex and let $x \in \partial \Omega$. We start by considering an infinitesimal hyperplane segment $\partial \Omega_x$ centered around $x$. By convexity of $\Omega$, almost all lines intersecting $\partial \Omega_x$ will intersect $\partial \Omega$ in exactly one other point. This implies, as the size of $\partial \Omega_x$ tends to 0, that
$$ \int_{L} n_{\ell}( \partial \Omega_x) n_{\ell}(\partial \Omega \setminus \partial \Omega_x) d\mu(\ell) = (1+o(1)) \cdot \mu(\partial \Omega_x).$$
At the same time, by Crofton's formula, the likelihood of a line hitting $\partial \Omega_x$ is only a function of the surface area of $\partial \Omega_x$ and independent of everything else. Finally, using linearity, we can decompose $\partial \Omega \setminus \partial \Omega_x$ into small hyperplane segments and use the computation above to deduce that
$$ \int_{\partial \Omega}  \frac{\left|\left\langle n(x), y - x \right\rangle  \left\langle  y - x, n(y) \right\rangle   \right| }{\|x -y\|^{n+1}}~d\sigma(y)  = \mbox{const}.$$
Integrating once more and applying the Theorem immediately implies that the constant has to be $c_n$. As for the second part, we can consider an infinitesimal hyperplane segment $H_x$ centered at $x \in \Omega \setminus \partial \Omega$ with normal direction given by $w \in \mathbb{S}^{n-1}$. Every line hitting $H_x$ intersects $\partial \Omega$ in exactly two points and thus
$$ \int_{L} n_{\ell}( H_x) n_{\ell}(\partial \Omega) d\mu(\ell) = 2 \cdot \mu(H_x).$$
By Crofton's formula, the right-hand side does not depend on the shape or location of $H_x$ and is only a function of the surface area of the infinitesimal segment. As for the left-hand side, using the computation done in the proof of the Theorem shows
 $$ \int_{L} n_{\ell}( H_x) n_{\ell}(\partial \Omega) d\mu(\ell) = (1+o(1))  \int_{\partial \Omega}  \frac{\left|\left\langle w, y - x \right\rangle  \left\langle  y - x, n(y) \right\rangle   \right| }{\|x -y\|^{n+1}}~d\sigma(y)$$
where the error term is with respect to the diameter of $H_x$ shrinking to 0.
\end{proof}

\end{document}